\documentclass[letterpaper, 10 pt, conference]{ieeeconf}  

\IEEEoverridecommandlockouts      

\usepackage{amsmath}
\usepackage{amsfonts}
\usepackage{amssymb}

\usepackage{graphicx}

\newtheorem{theorem}{Theorem}

\newtheorem{condition}[theorem]{Condition}

\newtheorem{definition}{Definition}
\newtheorem{problem}{Problem}

\usepackage[normalem]{ulem}

\title{\LARGE \bf
Transverse Contraction Criteria for Stability of \\Nonlinear Hybrid Limit Cycles
}

\author{Justin Z. Tang and Ian R. Manchester  
\thanks{*This work was supported in part by the Australian Research Council.}
\thanks{The authors are with the Australian Centre for Field Robotics (ACFR), Department of Aerospace, Mechanical and Mechatronic Engineering, University of Sydney, NSW 2006, Australia
        {\tt\small \{j.tang i.manchester\}@acfr.usyd.edu.au}}%
}

\begin{document}

\maketitle
\thispagestyle{empty}
\pagestyle{empty}

\begin{abstract}

In this paper, we derive differential conditions guaranteeing the orbital stability of nonlinear hybrid limit cycles.  These conditions are represented as a series of pointwise linear matrix inequalities (LMI), enabling the search for stability certificates via convex optimization tools such as sum-of-squares programming. Unlike traditional Lyapunov-based methods, the transverse contraction framework developed in this paper enables proof of stability for hybrid systems, without prior knowledge of the exact location of the stable limit cycle in state space. This methodology is illustrated on a dynamic walking example.
\end{abstract}

\section{Introduction}

Nonlinear hybrid dynamical systems with periodic solutions are widely found in diverse engineering and scientific fields such as electronics and mechanics. These hybrid systems contain continuous-time and discrete-time dynamics which interact with each other.   Stability of these dynamical systems is often a fundamental requirement for their practical value in applications. 

In this paper, we address the question: do all solutions of a hybrid nonlinear system starting in a particular set $K$ converge to a stable unique limit cycle?

A major motivation of this work is the study of underactuated bipedal locomotion \cite{Collins2005}, which can be represented as limit cycles in the state space \cite{Westervelt2007}. The control design and stability analysis of these ``dynamic walkers" are difficult since their dynamics are inherently hybrid and highly nonlinear \cite{Shiriaev2008,Manchester2010}.

The most well-known stability analysis tool for limit cycles is the Poincar\'e map \cite{Guckenheimer1997}, which describes the repeated passes of the system through a single transversal hypersurface. 
However, for nonlinear systems, the Poincar\'e map generally cannot be found explicitly. Further, since the system's evolution is only analyzed on a single surface, regions of stability in the full state space are difficult to evaluate. 

In practice, stability in the full state space is often estimated using exhaustive simulation, such as via cell-to-cell mapping  \cite{Hsu1980}, which has been applied to analysis of walking robots \cite{Schwab2001}.  However, computational costs of these methods are exponential in the dimension of the system.

In recent years, convex optimization methods have been widely applied in search for a ``stability certificate" based on Lyapunov theory \cite{Khalil2002}.  To characterize regions of stability for limit cycles, \cite{Goncalves2003} and \cite{Goncalves2005} introduced the notion of the Surface Lyapunov function, which verifies stability based on the ``impact map" between one switching surfaces to the next switching surface. The method is limited to Piecewise Linear Systems. In \cite{Manchester2010b}, nonlinear limit cycle stability analysis was performed by constructing Lyapunov functions in the transverse dynamics.  

However, these Lyapunov based methods require knowledge of the exact location of the limit cycle in state space, and hence are not applicable when the system dynamics are uncertain, since uncertainty will generally change the location of the limit cycle.

An alternative approach to Lyapunov methods is to search for a contraction metric  \cite{Aylward2008,Lohmiller1998}.  By defining stability incrementally between two arbitrary nearby trajectories, contraction analysis answers the question of whether the limiting behaviour of a given dynamical system is independent of its initial conditions. For analysis of limit cycles, transverse contraction was first introduced in \cite{Manchester2014}.

In this paper, we propose a transverse contraction framework for analysis of hybrid limit cycles, building on the work of transversal surface construction in  \cite{Manchester2010b}, and continuous transverse contraction of  \cite{Manchester2014}.  For the purposes of robustness analysis, an important advantage is that Lyapunov functions must be generally constructed around a known equilibrium, whereas a contraction metric derived herein implies existence of a stable equilibrium indirectly.  This is vital if the equilibrium point may change location depending on  unknown dynamics.

This paper proceeds as follows. Problem formulation and preliminary notations are outlined in Section II.  In Section III, the transverse contraction conditions guaranteeing stability of a limit cycle in a nonlinear hybrid system is presented.   We then formulate convex criteria enforcing these conditions on the nonlinear system in Section IV, thereby enabling the search for stability certificates via convex optimisation techniques such as sum-of-squares programming. An  illustrative example is given with a dynamic walking model in Section V.
Concluding remarks are given in Section VI.

\section{Preliminaries and Problem Formulation}


We consider the following class of autonomous hybrid dynamical problems.

\begin{problem}
Consider a hybrid system with impulse.
\begin{align}
\dot x &=f(x) &x\notin S_i^- \label{eq:hybrid-sys1}\\
x^+&=g (x) &x\in S_i^- \label{eq:hybrid-sys2}
\end{align}
where $f$, $g$ are smooth, and $S_i$ for $i=1,2...$ is defined as a ``switching
surface."
\end{problem}


We assume that $x\in \mathbb R^n$ and  $f:K\rightarrow \mathbb R^n$, where a set $K$ is a compact subset of $\mathbb R^n$ and strictly forward invariant under $f$, such that any solutions of (\ref{eq:hybrid-sys1}), (\ref{eq:hybrid-sys2}) starting in $x(0)\in K$ is in the interior of $K$ for all $t>0$.  We will refer to the Jacobian of $f$ as $A(x) := \frac{\partial f}{\partial x}$.

We denote the solution curve of the system as $\Phi (x_0,t)$, such that $x(t)=\Phi(x_0,t)$ is the solution at time $t>0$ of the dynamical system with initial state $x(0)=x_0$.

Suppose the system exhibits a non-trivial $T$-periodic orbit, i.e., for a periodic solution $x^*$, there exists some $T>0$ such that $x^\star(t)=x^\star(t+T)$ for all $t$.  Such a solution cannot be asymptotically stable, as perturbations in phase are persistent.  Instead, \textit{orbital stability} is better posed  \cite{Hale1980}.

The orbit of a periodic solution is the set $\mathcal X^\star := \{x\in \mathbb R^n : \exists t \in [0,T) : x=x^\star(t) \}$. The solution is said to be \textit{orbitally stable} if there exists a $b>0$ such that for any $x(0)$ satisfying $\text{dist}(x(0), \mathcal X^\star)<b$, the unique solution exists  and $\text{dist}(\Phi(x_0,t),\mathcal X^\star) \rightarrow 0$ as $t \rightarrow \infty$.  Further, the system is \textit{exponentially orbitally stable} if the system is orbitally stable and there exists a $b>0, \lambda >0, k>0$ such that $\text{dist}(\Phi(x_0,t),\mathcal X^\star) \leq k \text{dist}(\Phi(x_0,t),\mathcal X^\star)e^{-\lambda t}$. 

A switching surface defined by $S:=\{ x\ |\ c(x)=0\}$ is a $(n-1)$-dimensional hyperplane embedded in a manifold $\mathcal M$, where $c(x)$ is linear in $x$. The tangent space of $\mathcal M$ at $x \in \mathcal M$ is denoted as $T_x \mathcal M$, and the tangent bundle of $\mathcal M$ is denoted as $T \mathcal M = \bigcup_{x\in \mathcal M}\{x\} \times T_x\mathcal M.$

For simplicity, we assume that the region $K$ is broken  into a finite sequence of continuous ``tubes" in the state space, which contains no equilibrium point, i.e., $\forall x\in K, f(x)^{T}f(x)>0 $. Further, assume that all solutions starting in each continuous phase of $K$ approach a particular switching surface and that $\forall x \in S_i, f(x)^Tz_i(x)\neq 0$ where $z_i(x)$ is the normal vector of $S_i$.

Our goal is to verify that, in our hybrid system (\ref{eq:hybrid-sys1}), (\ref{eq:hybrid-sys2}), all solutions starting in the particular region $K$ are orbitally stable and converge to a unique limit cycle.  This is illustrated in Fig \ref{fig:tube} for a system with two continuous phases and two switching surfaces.  The verified region is shaded green, and the stable limit cycle in red.  Continuous dynamics are shown in solid line with discrete impulse between switching surfaces shown in dotted line.

\begin{figure}[h]
    \centering
    \includegraphics[width=4.8cm]{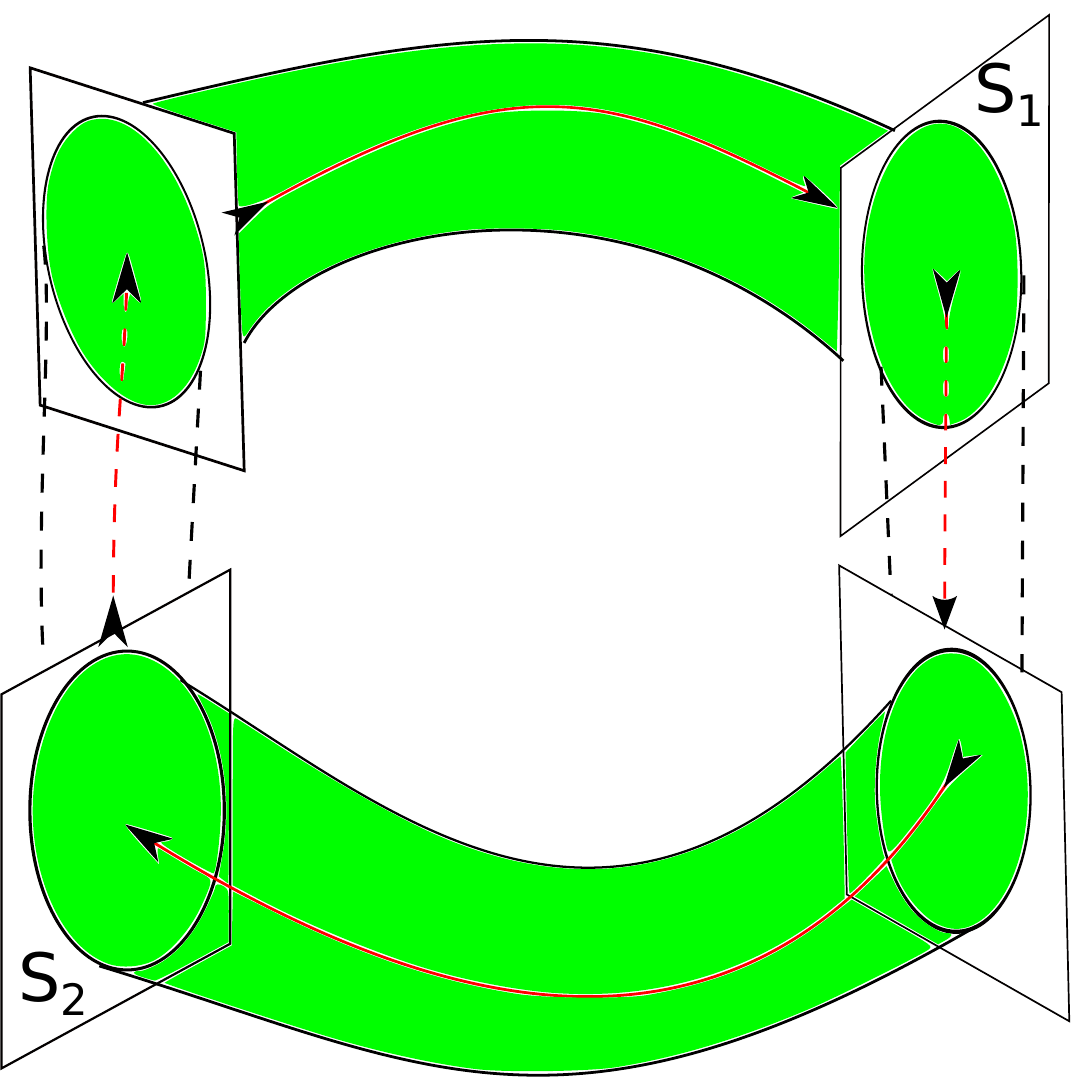}
    \caption{Region of stability around a hybrid limit cycle}
    \label{fig:tube}
\end{figure}

For completeness, we now restate the transverse contraction condition for continuous systems derived in \cite{Manchester2014}. Unless otherwise stated, we assume a Riemannian distance $V(x,\delta_x) = \sqrt{\delta_{x}'M(x)\delta_x}$,
with $M(x)$ symmetric positive-definite for all $x$.

\begin{definition}[Transverse Contraction]\label{def:transverse}
A continuous system $\dot x = f(x)$ is transverse contracting with rate $\lambda$ if there exists a Riemannian metric $V(x, \delta_x)$ satisfying

\begin{equation} \label{eq:ian-contract}
\frac{\partial V(x,\delta_x)}{\partial x} f(x) + \frac{\partial V(x,\delta_{x})}{\partial \delta_x} \frac{\partial f(x)}{\partial x}\delta_x \leq -\lambda V(x,\delta_x)
\end{equation}
for all $\delta_x\neq 0$ such that $\frac{\partial V}{\partial \delta_x}f(x)=0$.
\end{definition}

The latter condition requires $\delta_x$ to be transverse to the flow of the system, i.e., $\delta_ x$ and $f(x)$ are orthogonal with respect to the metric $M(x)$. In the case that $V(x,\delta_x) := \sqrt{\delta_x'M(x)\delta_x}$, it is true if $\delta_x'M(x)f(x)= 0$.

We now state a convex condition that is necessary and sufficient for transverse contraction derived in \cite{Manchester2014}.

\begin{definition}[Convex Criterion for Transverse Contraction] \label{def:contraction-metric-subsystem}
A system $\dot x = f(x)$ is transverse contracting with rate $\lambda$ and a metric $V(x,\delta) = \sqrt{\delta 'M(x)\delta}$ if and only if there exists a function $W(x):=M(x)^{-1}$ and $\rho(x) \geq 0$ such that
\begin{equation} \label{eq:transverse-contraction}
W(x)A(x)'+A(x)W(x)-\dot W(x) + 2\lambda W(x) - \rho(x)Q(x) \leq 0
\end{equation} 
where $Q(x):=f(x)f(x)'$ 
\end{definition}

It was shown in \cite{Manchester2014} that if a system is transverse contracting,
then all solutions starting with $x(0)\in K$ are stable under time reparameterization,
or ``Zhukovski stable" \cite{Leonov2006}, and hence converge to a unique
limit cycle.



\section{ Contraction Conditions for Limit Cycles in Hybrid Systems}

In this section, we derive the transverse contraction conditions for hybrid nonlinear limit cycles and show that these conditions guarantee the stability of a unique limit cycle within a particular set $K$. 

For simplicity of expression we will consider the problem of single switching surface and a single set of continuous dynamics. However, extension to multiple switches and multiple continuous phases is straightforward.

\begin{condition}[Metric condition] \label{th:metric-cond}
For a nonlinear system $\dot x = f(x)$ with a flat switching surface $S:=\{ x\ |\ c(x)=0\}$ if for all $x \in S$, the tangent space of the switching surface at $x$, $T_x\mathcal M$ satisfies
\begin{equation} \label{eq:metric-orthog}
f(x)^TM(x)\delta_x = 0
\end{equation}
for all $\delta_x \in T_x\mathcal M$, then, according to the metric $M(x)$, all trajectories approach the switching surface orthogonally. This is true since the direction vector of for any trajectory on the switching surface $x \in S$ is given by $f(x)$.
\end{condition}




\begin{theorem} \label{th:ball-contract}
If, firstly, the continuous part of the system in
Eq. (\ref{eq:hybrid-sys1}) is transverse contracting according to Definition
\ref{def:transverse}; secondly, Condition \ref{th:metric-cond}
is satisfied on the switching surface; and, thirdly, the discrete part of the system in Eq. (\ref{eq:hybrid-sys2}) satisfies 
\begin{equation}
\frac{\partial g}{\partial x}' M \frac{\partial g}{\partial
x} - M \leq 0
\end{equation}
 for all $\delta_x'Mf=0$; then,  all solutions in $K$ are locally Zhukovski stable.  Hence, we can construct, 
for each $x\in K$ locally a ball, $B_x$, of constant radius centred around a given trajectory at $x$, where trajectories starting in $B_x$ remain in $B_x$ under time reparameterization as $t\rightarrow \infty$.  

\end{theorem}

\begin{proof}
From the first condition in the Theorem and by the results of Definition \ref{def:transverse}, we know that the system is transverse contracting -- i.e., all virtual displacements $\delta_x$ in the subspace defined by the plane  $\delta_x' M(x) f(x) = 0$ are contracting.

Suppose we define locally for each $x\in K$ a smooth change of coordinates $x\rightarrow (\tau,x_\perp)$, highlighting the dynamics tangential and transverse to the flow of the system.

The transformation, $\Pi(x)$, can be represented by a set of bases $\{e_1,...,e_n\}$, where $e_1$ is in the direction of $f(x)$ and $e_2,...,e_n$ are independent and lie on the contracting plane defined by $\delta_x'M(x)f(x)=0$.  By definition of orthogonality, this also implies $e_2,...e_n$ are orthogonal to $e_1$. This transformation is given by:

\begin{equation} \label{eq:transform-coord}
\Pi(x)x = \begin{bmatrix} e_1 \\ e_2 \\ \vdots \\ e_n \end{bmatrix} x = \begin{bmatrix} \tau \\ x_\perp \end{bmatrix}
\end{equation}
where $\tau$ is a 1-dimensional ``phase" variable tangential to the flow of the system, and $x_\perp$ is the $(n-1)$-dimensional transverse dynamics in the contracting subspace.

The corresponding virtual displacements of the system in the new coordinate can be found via the Jacobian of the transformation.
\begin{equation} \label{eq:diff-coord-change}
\begin{bmatrix} \delta_\tau \\ \delta_\perp \end{bmatrix} = \bar \Theta \delta_x
\end{equation}
where $\bar \Theta := \frac{\partial \Pi}{\partial x}$. $\delta_\perp$ lies in the subspace defined by $\delta_x M(x) f(x) = 0$.

It is shown by  \cite{Hauser1994, Manchester2010b} that the differential system in the new coordinate becomes 

\begin{equation}
\frac{d}{dt} \begin{bmatrix} \delta_\tau \\ \delta_\perp \end{bmatrix} =
\begin{bmatrix} 0 & \star \\ 0 & A_\perp(x) \end{bmatrix} \begin{bmatrix}
\delta_\tau \\ \delta_\perp \end{bmatrix}
\end{equation}
where $A_\perp$ is the transverse linearization.


By the construction of $\Pi(x)$ in (\ref{eq:transform-coord}), since $e_1$ is orthogonal to all $e_2,...,e_n$, we have $e_1' M(x)e_k = 0$ for all $k=2,...,n$. Therefore, we can separate the transverse and tangential components in the metric $M(x)$ as below:
\begin{align}
V &= \delta_x' M(x) \delta_x \\
V &= \begin{bmatrix} \delta_\tau \\ \delta_\perp \end{bmatrix}^T 
\begin{bmatrix} M_\tau(x) & 0 \\ 0 & M _\perp (x) \end{bmatrix} 
\begin{bmatrix} \delta_\tau \\ \delta_\perp \end{bmatrix} \\
V &= M_\tau(x) |\delta_\tau|^2 + \delta_\perp M_\perp(x) \delta_\perp
\end{align}

Since $\delta_\perp$ lies entirely on the plane defined by $\delta_x M(x) f(x) = 0$, which is contracting by Definition \ref{def:transverse}, we yield: 
\begin{equation} \label{eq:perp-contract}
\delta_\perp \left(A_\perp(x)'M_\perp(x) +\ M_\perp(x) A_\perp(x)+ \dot M_\perp(x) \right) \delta_\perp <0
\end{equation}

By the results of \cite{Hauser1994}, we can construct a Lyapunov function $V_\perp = x_\perp' M_\perp x_\perp$ and for sufficiently small $\|x_\perp\|$  around the coordinate change, $\frac{d}{dt}V_\perp <0$ can be guaranteed by (\ref{eq:perp-contract}).


Now, by the second condition in the Theorem, during discrete instantaneous switching when $x\in S$, all trajectories approach the switching surface orthogonally  by the results of Condition \ref{th:metric-cond}. Hence, the transversal plane $\delta_xM(x)f(x)=0$ for trajectories on the switching surface aligns with the switching surface itself.  Suppose now on the switching surface, $V$ is non-increasing during the impulse:
\begin{align}
\delta_x'^+ M(x) \delta_x^+ &\leq \delta_x'^- M(x) \delta_x^- \\
\delta_x' \left(\frac{\partial g}{\partial x}' M \frac{\partial g}{\partial
x} - M \right) \delta_x &\leq 0 \label{eq:discrete-decrease}
\end{align}
for all $\delta_x M(x)f(x) = 0$.

Since the transversal plane for the trajectory coincide with the switching surface, by construction in (\ref{eq:transform-coord}), $\delta_\perp$ lies entirely on the switching surface. Hence (\ref{eq:discrete-decrease}) is equivalent to $V_\perp^+ \leq V^-_\perp$.

We now prove local Zhukovski stability using construction similar to \cite{Hauser1994}. Let $\sigma$ be a nearby solution to $x$.  Then, by the construction of $x_\perp$, there exists some compact region around $x$ where $k_1 \text{dist}(\sigma,x) \leq \|x_\perp\| \leq k_2\text{dist}(\sigma,x)$, for some $k_1, k_2>0$.  Hence, $k_3 \text{dist}(\sigma,x) \leq V_\perp \leq k_4\text{dist}(\sigma,x)$ for some $k_3,k_4>0$.

Since our conditions show $V_\perp$ is uniformly decreasing for all $x\in K$ in both the continuous dynamics and across the switching surface, as $t\rightarrow \infty$, $x_\perp \rightarrow 0$ locally.

Therefore, for all $x\in K$, there exists locally a ball, $B_x$, of constant radius centred around the trajectory at $x$, where trajectories
starting in $B_x$ remain in $B_x$ under time reparameterization as $t\rightarrow
\infty$. \end{proof}

\begin{theorem} \label{th:overall-contract}
If the conditions of Theorem \ref{th:ball-contract} is satisfied, for every pair of solutions $x_1$ and $x_2$ in $K$, there exists time reparameterization
$\tau(t)$ such that $x_1(t) \rightarrow x_2(\tau(t))$ as $t\rightarrow \infty$.
\end{theorem}

\begin{proof}
Suppose there are two nearby trajectories $x_1,x_2$.  We define a smooth
mapping $\gamma : [0,1] \rightarrow K$ where $\gamma(0)=x_1$ and $\gamma(1)
= x_2$, such that $\frac{\partial \gamma}{\partial s} \neq 0$ for all $s$.

Using the Riemannian metric $M(x)$ and associated distance function $V(x,\delta)
= \sqrt{\delta'M(x)\delta}$, the length of a smooth path between $x_1$ and
$x_2$ is given by
\begin{equation} \label{eq:line-integral}
L(\gamma)=\int_{0}^{1}V\left(\gamma(s),\frac{\partial }{\partial s} \gamma
(s) \right) ds
\end{equation}
Let $\Gamma(x_1,x_2)$ be the set of all smooth paths between $x_1$ and $x_2$.
Then, the geodesic distance between $x_1$ and $x_2$ is given by
\begin{equation} \label{eq:geodesic}
d(x_1,x_2)=\min_{\iota \in \Gamma(x_1,x_2)} L(\iota)
\end{equation} 
and $\iota(s)$ is the geodesic curve.

We prove, by contradiction, that Theorem \ref{th:ball-contract} proves Zhukovski stability between any $x_1,x_2\in K$.

Suppose  $x_1$ and $x_2$ diverge under some time reparameterization.  Then there exists a supremum $k$, where $k\in [0,1]$, for which $\iota(k)$ no longer converges to $x_1$.  But by Theorem \ref{th:ball-contract}, since $\iota(k)\in K$, one can construct at $\iota(k)$ a local ball of constant radius where immediately nearby trajectories
would remain in that ball after the impulse, which contradicts the proposition. 
Therefore,  $x_1(t) \rightarrow x_2(\tau(t))$ as $t\rightarrow \infty$.
 \end{proof}

\begin{theorem} If all conditions of Theorem \ref{th:ball-contract} are satisfied, there exists a unique limit cycle that is orbitally stable. 

\end{theorem}
\begin{proof}
Since $K$ is strictly forward invariant and compact, it follows that the omega-limit set, $\Omega(x)$, exists and is a compact subset of $K$. Further, an implication of Theorem \ref{th:overall-contract} is that all points in $K$ have the same $\omega$-limit set, which we denote $\Omega(K)$.

Pick a point $x^\star$ in $\Omega(K)$, by strict forward invariance, this is an interior point of $K$. Assume that $f(x^\star) \neq 0$, otherwise the results of \cite{Lohmiller1998} prove convergence to an equilibrium.  Construct a hyperplane orthogonal to $f(x^\star)$, which we denote by $H$.  We prove convergence to a limit cycle by constructing a Poincar\'e map on $H$.

Since $f(\cdot)$ is smooth, for $x$ in some neighbourhood $B$ of $x^\star$ we have that $f(x)'f(x)>0$, so in $B_H:=B \cap H$ solution curves are transversal to $H$ and pass through it in the same direction as at $x^\star$.

Since $x^\star$ is in the $\omega$-limit set for all points in $K$, and $B_H$ is transversal, the evolution of the system from any point $x(t)\in B_H$ eventually passes through $B_H$ again. That is, $x(t+s)\in B_H$ where $s>0$ depends on $x$.  This evolution can be represented by a Poincar\'e map $T:B_H \rightarrow B_H$.  

Take the distance between two points $d(x_1,x_2)$ in $B_H$ to be the Riemannian metric distance from Theorem \ref{th:overall-contract}.  By Theorem \ref{th:overall-contract}, we have that $d(T(x_1),T(x_2)) < d(x_1,x_2)$. Hence, $T$ is a contractive map from $B_H$ unto itself.  By the Banach fixed point theorem it has a unique stable fixed point, which is its only limit point so must be $x^\star$.  By standard results on Poincar\'e maps this implies that $x^\star$ is a point on a limit cycle, to which all solutions converge, by Theorem \ref{th:overall-contract}.  \end{proof}

\section{Convex Criteria for Limit Cycle Stability in Hybrid Systems}

In this section, we give convex conditions for transverse contraction of hybrid systems, enabling the search for the metric via sum-of-squares programming. 
\begin{condition}[Metric Condition linear in $W$] \label{theorem:metric-linear}
Suppose the normal vector of the switching surface is $z(x)$.  If the Riemannian
metric $M$ and the continuous dynamics $f$ satisfies\begin{equation} \label{eq:W-linear-cond}
\alpha (x)f(x) - W(x)z(x) = \beta (x) c(x)
\end{equation}
for some scalar function $\alpha(x) \geq 0$, 
then Condition \ref{th:metric-cond} is satisfied and all trajectories approach the switching surface orthogonally.
\end{condition}

\begin{proof}
For the orthogonality condition of (\ref{eq:metric-orthog}) in Condition \ref{th:metric-cond} to hold, we require $f'M(x)\delta_x = 0$ to hold for all $z(x)'\delta_x = 0$.  This is equivalent to requiring
for some scalar $\alpha(x)>0$, the following holds
\begin{equation}
z(x)'=\alpha(x) f(x)'M(x)
\end{equation}
for all $x \in S$.  Reformulating this in terms of $W:=M^{-1}$ we yield the requirement\begin{equation}
\alpha(x) f(x) = W(x)z(x)
\end{equation}
for all $c(x) = 0$.  Using an S-procedure formulation, we yield the equivalent equality constraint:

\begin{equation}
\alpha(x) f(x) -W(x)z(x) = \beta(x) c(x)
\end{equation}
which is the required condition for all $x$.  \end{proof}

\begin{theorem}[Convex Conditions for Limit Cycle Stability in Hybrid Nonlinear Systems]
\label{th:hybrid-metric}
Suppose, firstly, there exists a Riemannian metric $M(x)$ for nonlinear system (\ref{eq:hybrid-sys1})
and (\ref{eq:hybrid-sys2}) which satisfies Remark \ref{theorem:metric-linear};
secondly, the continuous dynamics of the system satisfies 
\begin{equation}
W(x)A(x)'+A(x)W(x)-\dot W(x) + 2\lambda W(x) - \rho(x)Q(x) \leq 0
\end{equation} 
where $W(x):=M^{-1}, Q(x):=f(x)f(x)'$,  and thirdly, the discrete switching dynamics
$g(x)$ of the system satisfies the following LMI

\begin{equation} \label{eq:schur-lmi}
\begin{bmatrix} W(x)+\zeta (x)Q(x) & W(x)\frac{\partial g}{\partial x}^T
\\
\frac{\partial g}{\partial x}W (x) & W(x) \\
\end{bmatrix} \geq 0
\end{equation}

For some $\zeta(x) \geq 0$, then the overall hybrid system is contracting
with respect to metric $M$.
\end{theorem}

\begin{proof}
By Remark \ref{theorem:metric-linear}, just prior to switching, all trajectories approach the switching surface $S$ orthogonally. 

During the discrete jump, we require the following
condition to be satisfied\begin{align}
\delta_x' \left(\frac{\partial g}{\partial x}' M \frac{\partial g}{\partial
x} - M \right) \delta_x &\leq 0 
\end{align}
for all $\delta_x'Mf=0$.

Reformulating in terms of the gradient of the metric, i.e. $\eta := M(x)
\delta_x$ such that $\delta = M^{-1} \eta := W\eta$, we yield the equivalent
condition:
\begin{align}
\eta' \left( W \frac{\partial g}{\partial x}^T W^{-1} \frac{\partial
g}{\partial x} W - W\right) \eta &\leq 0
\end{align}

The transversality condition $\delta'Mf=0$ becomes $\eta'f(x)=0$.  Now, define
matrix function $Q(x):=f(x)f(x)'$ which is rank-one and positive-semidefinite.
 Hence the sets $\{ \eta : \eta'f(x)=0\}$, $\{ \eta : \eta' Q \eta = 0\}$
and $\{ \eta : \eta'Q(x) \eta \leq 0\}$ are equivalent.

Now, the transverse contraction of the discrete switching can be proved by
the existence of $W(x)$ such that:
\begin{multline}
\eta'Q(x) \eta \leq 0 \Rightarrow \\
\eta' \left( W \frac{\partial g}{\partial x}^T W^{-1} \frac{\partial
g}{\partial x} W - W\right) \eta \leq 0 
\end{multline}

By the S-procedure, the condition is only true if and only if there exists $\zeta(x)\geq
0$ such that
\begin{equation} \label{eq:discrete-lmi}
\eta' \left(  W + \zeta(x) Q(x) - W \frac{\partial g}{\partial x}^T W^{-1}
\frac{\partial g}{\partial x} \right) \eta \geq 0
\end{equation}

By the Schur Complement, (\ref{eq:discrete-lmi}) is true if and only if (\ref{eq:schur-lmi})
holds, which completes the proof. 
\end{proof}
Note that these conditions are all linear in the unknown functions $W(x), \alpha(x),\beta(x),\rho(x),$ and $\zeta (x)$, i.e., it consists of a linear matrix inequality at each point $x$.  For polynomial systems, these conditions can be verified efficiently using sum-of-squares programming and positivstellensatz arguments \cite{Tan2006}.

\section{Application Example}

The rimless wheel is a simple planar model of dynamic walking, exhibiting hybrid (switching) behaviour.  It consists of a central mass, of mass $g,$ with equally spaced spikes, of length $l$ extending radially outwards. The system rolls down an incline of pitch $\gamma$, as shown in Fig. \ref{fig:rimless}.

\begin{figure}[h]
    \centering
    \includegraphics[width=4cm]{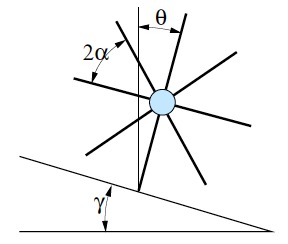}
    \caption{The rimless wheel model}
    \label{fig:rimless}
\end{figure}

At any given moment, the rimless wheel rotates about the stance foot without slipping, behaving like an inverted pendulum.  When the next foot contacts the ground, it is assumed that an elastic collision occurs such that the old stance foot lifts off and the system now rotates about the new stance foot.

The Rimless Wheel state space $x=[\theta,\dot\theta]'$ can be represented with the following hybrid system
dynamics:\begin{align}
\frac{d}{dt}
\begin{bmatrix} \theta\\ \dot\theta \end{bmatrix} &= f(\theta,\dot\theta) = \begin{bmatrix} \dot\theta \\ \frac{g}{l}\sin \theta \end{bmatrix}  
& \text{for } \theta - \gamma - \alpha \neq 0 
\\
\begin{bmatrix} \theta^+ \\ \dot\theta^+ \end{bmatrix} &= g(\dot\theta^- )= \begin{bmatrix} \gamma - \alpha \\ \cos (2\alpha) \dot\theta^- \end{bmatrix} &\text{for } \theta - \gamma-\alpha = 0 
\end{align}

On a sufficiently inclined slope, the system has a stable limit cycle, for which the energy lost in collision is perfectly compensated by the change in potential energy. 

The system has been studied extensively and its basin of attraction has ben computed exactly. \cite{Coleman1998}

\begin{figure}[h]
    \centering
    \includegraphics[width=\linewidth]{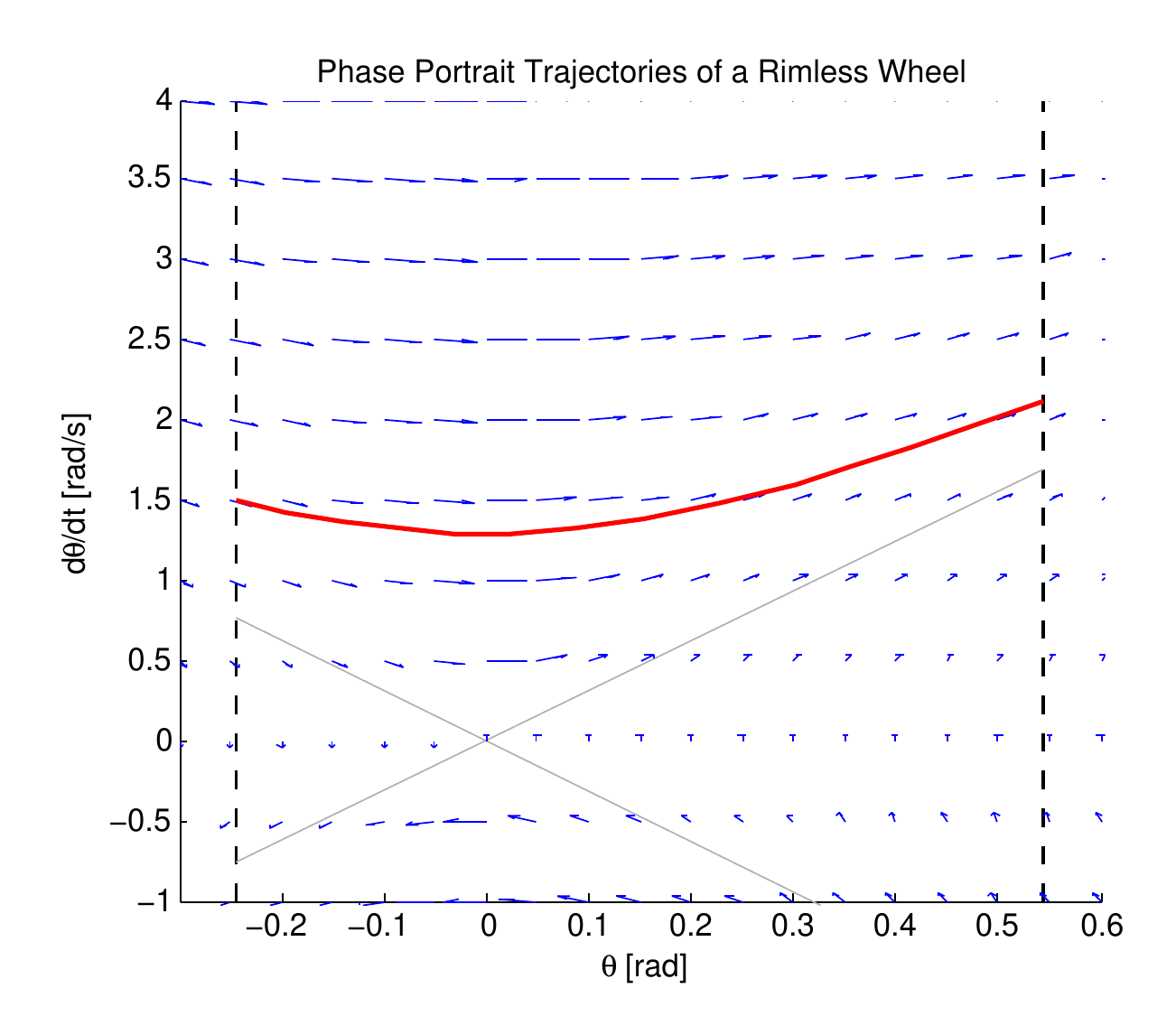}
    \caption{Phase diagram of the rimless wheel model}
    \label{fig:rimless-limit-cycle}
\end{figure}

Figure \ref{fig:rimless-limit-cycle} shows the phase portrait of the rimless
wheel, with blue arrows indicating the direction of the continuous dynamics.
 The dotted line on the right of the graph indicates the collision surface
that maps to the left edge of the graph (or vice-versa, depending on the
direction of dynamics). The grey line represents the homoclinic orbits of the system, and the red line represents the stable limit cycle.

Using the convex conditions of Theorem \ref{th:hybrid-metric}, we formulate sum-of-squares\ (SOS) and Positivestellansatz conditions \cite{Parrilo2003}
which verifies transverse contraction for the hybrid system in a region around the limit cycle, defined by the switching surfaces and a B\'ezier polynomial $b(x)$.

Let $H = A(x)W(x) + W(x)A(x)' - \dot W(x) + 2\lambda W(x)$, let $\Sigma_n[x]$ denote the set of $n \times n$ matrices verified positive semidefinite.
We approximate the continuous dynamics $f(x)$ with a third order taylor series expansion.

The conditions verified are given below.
\begin{align} 
\begin{split}
W(x) - (f(x)^T f(x)- \epsilon ) L_1(x)  \\ 
- (\theta - (\gamma - \alpha)) L_2(x) \\ 
- ((\alpha+\gamma)-\theta) L_3(x) \\
- (\dot\theta - b(x))L_4(x)
&\in \Sigma_n[x]
\end{split} \label{eq:positive-W}
\\[1em]
\alpha f(x) - W(x) \nabla c &= \beta (x)c(x) \label{eq:equal-sos} \\[1em]
\begin{split}
-H - \rho(x) f(x) f(x)' \\ 
- (f(x)^T f(x)- \epsilon ) L_5(x) \\
- (\theta - (\gamma - \alpha)) L_6(x) \\
- ((\alpha+\gamma)-\theta) L_7(x) \\
- (\dot\theta - b(x))L_8(x)
& \in \Sigma_n[x]
\end{split} \label{eq:sos-contract}
\\[1em]
\begin{bmatrix} W(x) + \zeta(x)Q(x) & W(x) \frac{\partial g}{\partial x}^T\\ \frac{\partial g}{\partial x} W(x) & W(x)\end{bmatrix}&\in \Sigma_{2n}[x] \label{eq:schur-cond}  \\[1em]
L_1, L_2, L_3, L_4, L_5, L_6, L_7, L_8, \beta(x) &\in \Sigma_n[x] \label{eq:lagrange-sos} \\[1em]
 \alpha(x),\rho(x) , \zeta(x) &\in  \label{eq:scalar-sos}  \Sigma[x]
\end{align}

(\ref{eq:positive-W}) verifies the positive-definiteness of $W$ within the defined region; (\ref{eq:equal-sos}) verifies the condition of Remark \ref{theorem:metric-linear}; (\ref{eq:sos-contract}) and (\ref{eq:schur-cond}) verifies the conditions of Theorem \ref{th:hybrid-metric}; and, finally, (\ref{eq:lagrange-sos}) and (\ref{eq:scalar-sos}) verifies positive semi-definiteness of the Lagrange multipliers and scalar functions.

The above conditions were formulated in YALMIP \cite{Lofberg2004,Lofberg2009} and solved by commercial SDP solver MOSEK v.7.0.0.103.  The code has been made available online \cite{rimless2014}.

We found that these conditions could be verfified with $W(x)$ and $\beta(x)$ a matrix of degree-four polynomials, and $L_i(x), \alpha(x), \zeta(x), \rho(x)$ degree-two.  Figure \ref{fig:rimless-phase-v} shows verified regions of stability coloured in green.

\begin{figure}[h]
    \centering
    \includegraphics[width=\linewidth]{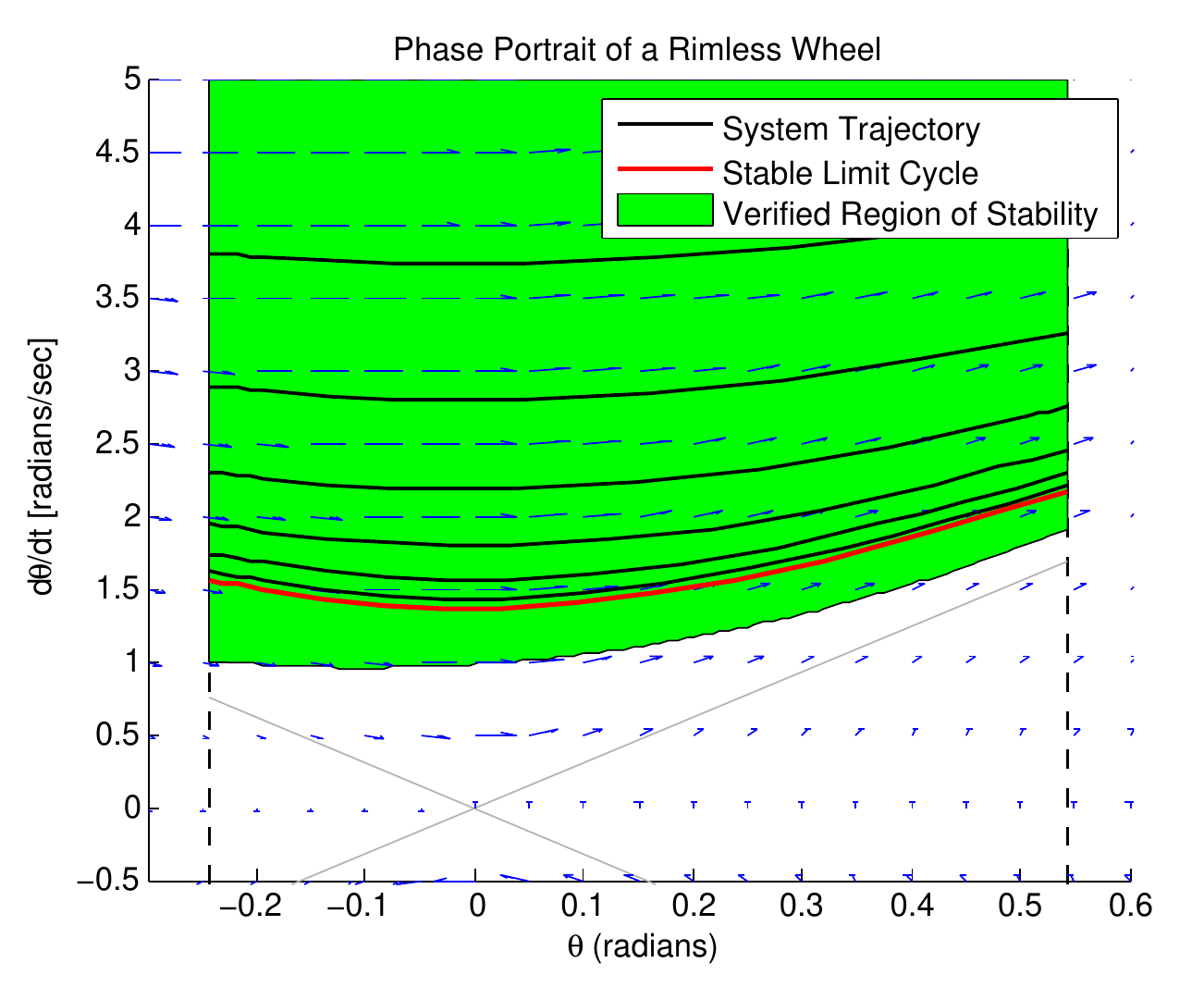}
    \caption{Verified region of transverse contraction for Rimless Wheel}
    \label{fig:rimless-phase-v}
\end{figure}

\section{Conclusion}
We have derived differential conditions guaranteeing the orbital stability of nonlinear hybrid limit cycles.  These conditions are presented as pointwise linear matrix inequalities,  enabling an efficient search  for a stability certificate.

The main advantages of this approach over traditional Lyapunov-based methods are two-fold.  Firstly, the transverse contraction framework decouples the question of convergence from knowledge of a particular solution.  This opens doors to robustness analysis when the exact location of the limit cycle is unknown due to uncertainty in the dynamics.

Further, this method simplifies the search for stability certificate compared with previous Lyapunov-based method in \cite{Manchester2010}, which requires a separate search for transversal surfaces and valid Lyapunov functions on those surfaces.  By encapsulating the direction of transversal surfaces with the definition of orthogonality, this method allows the search for stability certificate by a single convex optimization problem -- a search for a valid  transverse contraction  metric $M(x)$.

 The ability to efficiently compute stability
certificates for hybrid systems in this  work opens opportunities for control design with
guaranteed stabilizability \cite{Manchester2013}, and would enable the search
for provably stable models in system identification \cite{Tobenkin2010a,Manchester2011a}.

\bibliographystyle{IEEEtran}
\bibliography{library}

\end{document}